\documentclass[11pt,a4paper]{article}

\usepackage[T1]{fontenc}
\usepackage[utf8]{inputenc}
\usepackage{lmodern}
\usepackage[margin=25mm]{geometry}
\usepackage{mathtools,amssymb,amsthm}
\usepackage{graphicx}
\usepackage{enumitem}
\usepackage{microtype}
\usepackage[hidelinks]{hyperref}

\numberwithin{equation}{section}
\theoremstyle{plain}
\newtheorem{theorem}{Theorem}[section]
\newtheorem{proposition}{Proposition}[section]
\newtheorem{lemma}{Lemma}[section]
\newtheorem{corollary}{Corollary}[section]
\theoremstyle{definition}
\newtheorem{assumption}{Assumption}[section]
\newtheorem*{assumption*}{Assumption}
\theoremstyle{remark}
\newtheorem{remark}{Remark}[section]

\newcommand{\E}{\mathbb{E}}
\DeclareMathOperator{\Law}{Law}
\newcommand{\R}{\mathbb{R}}
\newcommand{\JJ}{\mathcal{J}}
\newcommand{\EE}{\mathcal{E}}
\DeclarePairedDelimiter{\abs}{\lvert}{\rvert}
\newcommand{\dd}{\mathop{}\!\mathrm{d}}
\newcommand{\coloneq}{\coloneqq}

\title{Optimal Time-Dependent Jump Truncation for Time-Singular L\'evy Processes}
\author{Victoria Knopova\textsuperscript{1,2} \and Denis Platonov\textsuperscript{1}}
\date{}
\hypersetup{
  pdftitle={Optimal Time-Dependent Jump Truncation for Time-Singular L\'evy Processes},
  pdfauthor={Victoria Knopova, Denis Platonov}
}
\begin{document}
\maketitle

\begin{center}
\small
\textsuperscript{1}Department of Probability Theory, Statistics and Actuarial Mathematics,\\
Faculty of Mechanics and Mathematics,\\
Taras Shevchenko National University of Kyiv,\\
Volodymyrska Str.~64, Kyiv 01601, Ukraine\\[1ex]
\textsuperscript{2}Institute of Mathematics, National Academy of Sciences of Ukraine,\\
Tereshchenkivska Str.~3, Kyiv-4 01024, Ukraine\\[1ex]
Victoria Knopova: \href{mailto:vicknopova@knu.ua}{\texttt{vicknopova@knu.ua}},
\href{mailto:vicknopova@imath.kiev.ua}{\texttt{vicknopova@imath.kiev.ua}}\\
Denis Platonov: \href{mailto:dplatonov@knu.ua}{\texttt{dplatonov@knu.ua}}
\end{center}

\begin{abstract}
We study optimal jump truncation for the additive time-singular pure-jump model of the form
$X_T=\int_0^T t^{-\sigma}\,\dd Z_t$, where $Z$ is a L\'evy process and $\sigma\ge 0$.
For a fixed expected jump cost, we minimize the residual small-jump variance over measurable time-dependent cutoffs. Under regularity and tail assumptions on the L\'evy measure, we prove that the problem admits an optimal cutoff, unique up to a.e. equality, of the form $r^\ast(t)=\bigl(c^\ast t^\sigma\bigr)\wedge 1$.
In the symmetric $\alpha$-stable case, we obtain explicit matched-cost comparisons with the classical fixed cutoff and show that within the admissible non-truncated regime the Dynamic Cutting family is strictly better than the classical fixed cutoff whenever $0<\sigma<1/2$. Finally, for symmetric L\'evy measures and cutoffs satisfying the corresponding $L^p$-integrability assumptions, we derive the weak-error bound $W_p(r)\le C_p\,\EE^{p/2}(r)$, $0<p<2$, demonstrating that any cutoff which minimizes the residual variance also minimizes the corresponding upper bound for the weak error.
\end{abstract}



\section*{Introduction}
Investigation of infinite-activity jump models requires introducing a truncation level below which the small jumps are discarded. During the past decades this subject was extensively addressed in the literature.
Convergence results for the Euler--Maruyama discretization scheme, in which small jumps are discarded, go back to Protter and Talay \cite{ProtterTalay1997} and Jacod \cite{Jacod2004}.
Asmussen and Rosi\'nski \cite{AsmussenRosinski2001} established conditions under which the compensated sum of small jumps can be effectively replaced by a Gaussian approximation. Rubenthaler \cite{Rubenthaler2003} studied Euler approximations that discard small jumps. Fournier \cite{Fournier2011} studied Gaussian replacement of the small-jump component and established improved convergence rates in law compared with omission. To further optimize the efficiency, Kohatsu-Higa and Tankov \cite{KohatsuHigaTankov2010} proposed jump-adapted discretization schemes that utilize non-uniform time grids based on the arrival times of large jumps. This methodology has been adapted for simulation of L\'evy processes and L\'evy-driven stochastic differential equations (SDEs). Extending the Asmussen--Rosi\'nski approach to time-inhomogeneous settings, Bossy and Maurer \cite{BossyMaurer2025} proposed the Euler--Maruyama scheme for SDEs driven by a time-inhomogeneous Poisson random measure and derived optimal strong and weak convergence rates.

In our recent papers we investigated time-dependent truncation, which we call \emph{Dynamic Cutting} (DC) (see \cite{IvanenkoKnopovaPlatonov2025,KnopovaPlatonov2026,Platonov2026}). Numerical experiments at matched computational cost show that DC yields consistently smaller $L^1$-strong errors than classical fixed truncation for the tested time-singular jump coefficients; see \cite{KnopovaPlatonov2026}. In \cite{Platonov2026} it is shown that the upper bound on the weak error in the Euler--Maruyama scheme with DC approximation, combined with compensation of the small-jump part by a suitably chosen Gaussian random variable, is of order $O(n^{-1})$. However, the DC method demonstrates some advantages when we study models with singular coefficients; see Section~\ref{sec:motivation}. On the other hand, the question of how to choose the optimal dynamic truncation level that balances residual variance and computational jump cost remains open. In this paper we provide an answer to this question in a simple model situation. We study an additive pure-jump model with time singularity (cf.~\eqref{eq:prototype-canonical} below) and compare truncation rules for jump simulation. We show that a fixed truncation level is suboptimal in the presence of a time singularity of the form $t^{-\sigma}$, $\sigma>0$.

The paper is organized as follows. Section~\ref{sec:motivation} outlines why such models are useful. Section~\ref{sec:setup} introduces the time-singular model and truncation functionals and states the main results. Proofs are given in Sections~\ref{sec:universal-shape}, \ref{sec:stable}, and \ref{sec:weak}.

\section{Motivation: L\'evy-driven Volterra models}\label{sec:motivation}

L\'evy-driven Volterra models with deterministic kernels are used in several applied areas, including energy finance, environmental risk, turbulence, etc. In particular, Barndorff-Nielsen, Benth, and Veraart \cite{BarndorffNielsenBenthVeraart2013} introduce volatility-modulated L\'evy-driven Volterra processes as a framework for energy spot prices, emphasizing their ability to incorporate stochastic volatility, jumps, spikes, and autocorrelation structures. Di Nunno, Fiacco, and Karlsen \cite{DiNunnoFiaccoKarlsen2019} likewise study L\'evy-driven Volterra processes (including Gamma--Volterra) with applications ranging from turbulence to energy finance.

These works motivate the study of truncation and approximation errors in the presence of a local kernel singularity. Below we provide a heuristic explanation of how these models are related to our investigation.

A broad class of jump-driven Volterra models can be written in the form
\begin{equation}\label{eq:general-volterra-motivation}
    Y_t = \int_0^t K(t,s) a(s)\,\dd Z_s,
\end{equation}
where $Z$ is a L\'evy process; without loss of generality assume that $a(s)$ is continuous. Suppose that near the diagonal $s=t$ the kernel has power-type behaviour
\begin{equation}\label{eq:kernel-pow}
    K(t,s) \sim c(t)(t-s)^{-\sigma},
    \qquad \sigma\in(0,1/2),
    \qquad s\uparrow t,
\end{equation}
Then for small $\Delta\in(0,t)$ one can approximate
\[
\int_{t-\Delta}^t K(t,s)a(s)\,\dd Z_s
\approx
c(t)a(t)\int_{t-\Delta}^t (t-s)^{-\sigma}\,\dd Z_s.
\]
Since a L\'evy process has independent stationary increments, by a standard procedure one can show that for fixed $t$
\[
\Law\left(\int_{t-\Delta}^t (t-s)^{-\sigma}\,\dd Z_s\right)
=
\Law\left(\int_0^\Delta s^{-\sigma}\,\dd Z_s\right).
\]

Therefore, the additive pure-jump model studied in this paper can be used in situations where simulation of the trajectory near the singularity requires extra attention. As we will see, the DC procedure ``amplifies'' the small jumps in the region of time singularity, making the method more precise than fixed-level cutting. This connection is especially visible for the Gamma--Volterra kernels studied by Di Nunno, Fiacco, and Karlsen \cite{DiNunnoFiaccoKarlsen2019}, where functionals of the type
\[
Y_t = \int_0^t g(t-s)\,\dd Z_s
\]
were investigated, especially with $g(t)=t^\beta e^{-\lambda t}$, $\beta\in(-1/2,1/2)$, $\lambda>0$. This is exactly our situation if $\beta=-\sigma\in(-1/2,0)$.

From this perspective, our optimization problem is not an isolated exercise. It identifies the truncation rule that is theoretically optimal for the local singular-jump component which appears inside broader Volterra models with power-type kernel singularities. It provides a benchmark for the design of dynamic cutting and approximations, which possibly can be extended to more general singular-kernel jump models.

\section{Setup and results}\label{sec:setup}

Let $Z_t$ be a L\'evy process with Poisson random measure $\mathcal{N}(\dd t,\dd z)$, compensator $\dd t\,\nu(\dd z)$, and compensated measure $\widetilde{\mathcal{N}}(\dd t,\dd z)$. Here $\nu$ is a L\'evy measure on $\R\setminus\{0\}$, i.e.
\begin{equation}\label{eq:levy-integrability}
    \int_{\R} (1\wedge z^2)\,\nu(\dd z)<\infty.
\end{equation}
Unless stated otherwise, we assume that $\nu$ is symmetric.

Let
\begin{equation}\label{eq:prototype-canonical}
    X_T \coloneq
    \int_0^T \int_{|z|\le 1} t^{-\sigma}z\,\widetilde{\mathcal N}(\dd t,\dd z)
    +
    \int_0^T \int_{|z|>1} t^{-\sigma}z\,\mathcal N(\dd t,\dd z), \quad T>0.
\end{equation}

Let $r:[0,T]\to(0,1]$ be a measurable cutoff function. We cut out small jumps in the representation of $X_T$ at the level $r(t)$ (since we intend to cut out only small jumps, we assume that $r(t)\le 1$ for all $t\in[0,T]$):
\begin{equation}\label{eq:retained-theoretical}
    X_T^{(r)} \coloneq
    \int_0^T \int_{r(t)<\abs{z}\le 1} t^{-\sigma} z\,\widetilde{\mathcal N}(\dd t,\dd z)
    +
    \int_0^T \int_{\abs{z}>1} t^{-\sigma} z\,\mathcal N(\dd t,\dd z).
\end{equation}
Denote by $R_T^{(r)}$ the omitted small-jump component
\begin{equation}\label{eq:residual-def}
    R_T^{(r)} \coloneq \int_0^T \int_{\abs{z}\le r(t)} t^{-\sigma} z\,\widetilde{\mathcal N}(\dd t,\dd z).
\end{equation}
Consequently,
\begin{equation}\label{eq:decomp}
    X_T = X_T^{(r)} + R_T^{(r)}.
\end{equation}

Define the following auxiliary functions:
\begin{equation}\label{eq:V-def}
   N(\rho) \coloneq \nu(\{|z|>\rho\}), \qquad V(\rho) \coloneq \int_{\abs{z}\le \rho} z^2\,\nu(\dd z),
    \qquad \rho>0.
\end{equation}
For a measurable function $\phi:[0,T]\to[0,1]$, define
\begin{equation}\label{eq:cost-error-functionals}
    \JJ(\phi) \coloneq \int_0^T N(\phi(t))\,\dd t,
    \quad
    \EE(\phi) \coloneq \int_0^T t^{-2\sigma} V(\phi(t))\,\dd t.
\end{equation}
We identify cutoffs up to equality almost everywhere.
Note that $N(r)$ is the intensity of the time-inhomogeneous Poisson process
\[
Y_T\coloneq \int_0^T \int_{|z|>r(t)} \mathcal{N}(\dd t,\dd z).
\]
Thus,
\begin{equation}\label{NY}
\JJ(r) = \E Y_T
\end{equation}
is the expected number of jumps of $Y$ on $[0,T]$. Whenever $\EE(r)<\infty$, the process $R^{(r)}$ is a square-integrable martingale, and the It\^o isometry (see, e.g., \cite{BottcherSchillingWang2013}) gives
\begin{equation}\label{eq:ito-isometry-error}
    \E|R_T^{(r)}|^2 = \E\langle R^{(r)}\rangle_T = \int_0^T \int_{\abs{z}\le r(t)} t^{-2\sigma} z^2\,\nu(\dd z)\,\dd t = \EE(r),
\end{equation}
which shows that $\EE(r)$ is the residual small-jump variance. Here $\langle R^{(r)}\rangle_T$ is the predictable quadratic variation of the martingale $R^{(r)}$ evaluated at time $T$. Below we provide a sufficient condition under which $\EE(r)$ is finite.

Our first result concerns the following deterministic optimization problem for a fixed simulation cost $J_0>0$:
\begin{equation}\label{eq:fixed-cost-problem}
    \min_{r}\; \EE(r)
    \qquad\text{subject to}\qquad
    \JJ(r)=J_0
\end{equation}
Since $r(t)\le 1$, by monotonicity of $N(\cdot)$ we have
\[
\JJ(r)\ge \int_0^T N(1)\,\dd t = TN(1),
\]
implying that $J_0\ge TN(1)$ is a necessary condition for the constrained problem \eqref{eq:fixed-cost-problem}.

\begin{assumption}\label{ass:global-opt}
     \leavevmode
     \begin{enumerate}[label=(A\arabic*),itemsep=2pt,topsep=4pt]
         \item $N\in C^1((0,1])$ and $N'(\rho)<0$ for $\rho\in(0,1]$.
         \item $N(0+)=\infty$.
         \item There exist constants $\alpha\in(0,2)$ and $C_\alpha>0$ such that
         \begin{equation}\label{eq:N-power-tail}
             N(\rho)\le C_\alpha \rho^{-\alpha},
             \qquad \rho\in(0,1].
         \end{equation}
     \end{enumerate}
\end{assumption}
Note that \eqref{eq:N-power-tail} always holds true with $\alpha=2$.

\begin{theorem}[Existence, uniqueness and global optimality]\label{thm:global-opt}
    Under Assumption~\ref{ass:global-opt}, the following hold.

    \begin{enumerate}[label=(\roman*),itemsep=2pt,topsep=4pt]
        \item If $\sigma\in[0,1/2)$ and $J_0=TN(1)$, then the constrained optimization problem \eqref{eq:fixed-cost-problem}
        has a unique minimizer
        \[
        r^\ast(t)\equiv 1
        \qquad
        t\in[0,T].
        \]

        \item If $\sigma\in[0,1/\alpha)$ and $J_0>TN(1)$, then the constrained optimization problem \eqref{eq:fixed-cost-problem} over measurable cutoffs $r:[0,T]\to(0,1]$ has a unique minimizer
        \begin{equation}\label{eq:global-rstar-final}
            r^\ast(t)=\bigl(c^\ast t^\sigma\bigr)\wedge 1,
            \qquad
            t\in(0,T],
        \end{equation}
        where $c^\ast$ is a unique constant such that
        \begin{equation}\label{eq:global-cost-final}
            \int_0^T N\bigl((c^\ast t^\sigma)\wedge 1\bigr)\,\dd t=J_0,
        \end{equation}
        If in addition $c^\ast T^\sigma\le 1$, then
        \[
        r^\ast(t)=c^\ast t^\sigma
        \qquad\text{for all }t\in(0,T].
        \]
    \end{enumerate}
\end{theorem}

We now consider the $\alpha$-stable case, which makes the general optimizer from Theorem~\ref{thm:global-opt} explicit and provides a setting for comparison with the fixed cutoff.
In this case the L\'evy measure is
\begin{equation}\label{eq:stable-levy}
    \nu(\dd z)=\alpha M\abs{z}^{-1-\alpha}\,\dd z,
    \qquad \alpha\in(0,2),\; M>0.
\end{equation}
Then $N(\rho)=2M\rho^{-\alpha}$.
For $\varepsilon\in[0,1)$ and $c>0$, consider the power-law cutoff family
\begin{equation}\label{eq:stable-family}
    r_{c,\varepsilon}(t)\coloneq c\,t^{\varepsilon/\alpha},
    \qquad 0<t\le T.
\end{equation}
In this model, the inverse-tail parametrization used by Dynamic Cutting is explicit. Indeed,
\[
\tau(t)\coloneq \sup\{\rho\ge 0: N(\rho)\ge 2/t\}
\implies
\tau(Ct^\varepsilon)=c\,t^{\varepsilon/\alpha}
\]
for a suitable constant $c>0$. Thus, in the $\alpha$-stable case, the DC family coincides with the power-law family \eqref{eq:stable-family}. Moreover, Theorem~\ref{thm:global-opt} suggests the exponent $\varepsilon=\alpha\sigma$, since then $r_{c,\varepsilon}(t)\propto t^\sigma$ before truncation at $1$. Theorem~\ref{thm:matched-cost} shows that this exponent is optimal within the admissible non-truncated class. Figure~\ref{fig:tau-viz} shows the family \eqref{eq:stable-family} for several values of $\varepsilon$.

\begin{figure}[htbp]
    \centering
    \includegraphics[width=0.7\textwidth]{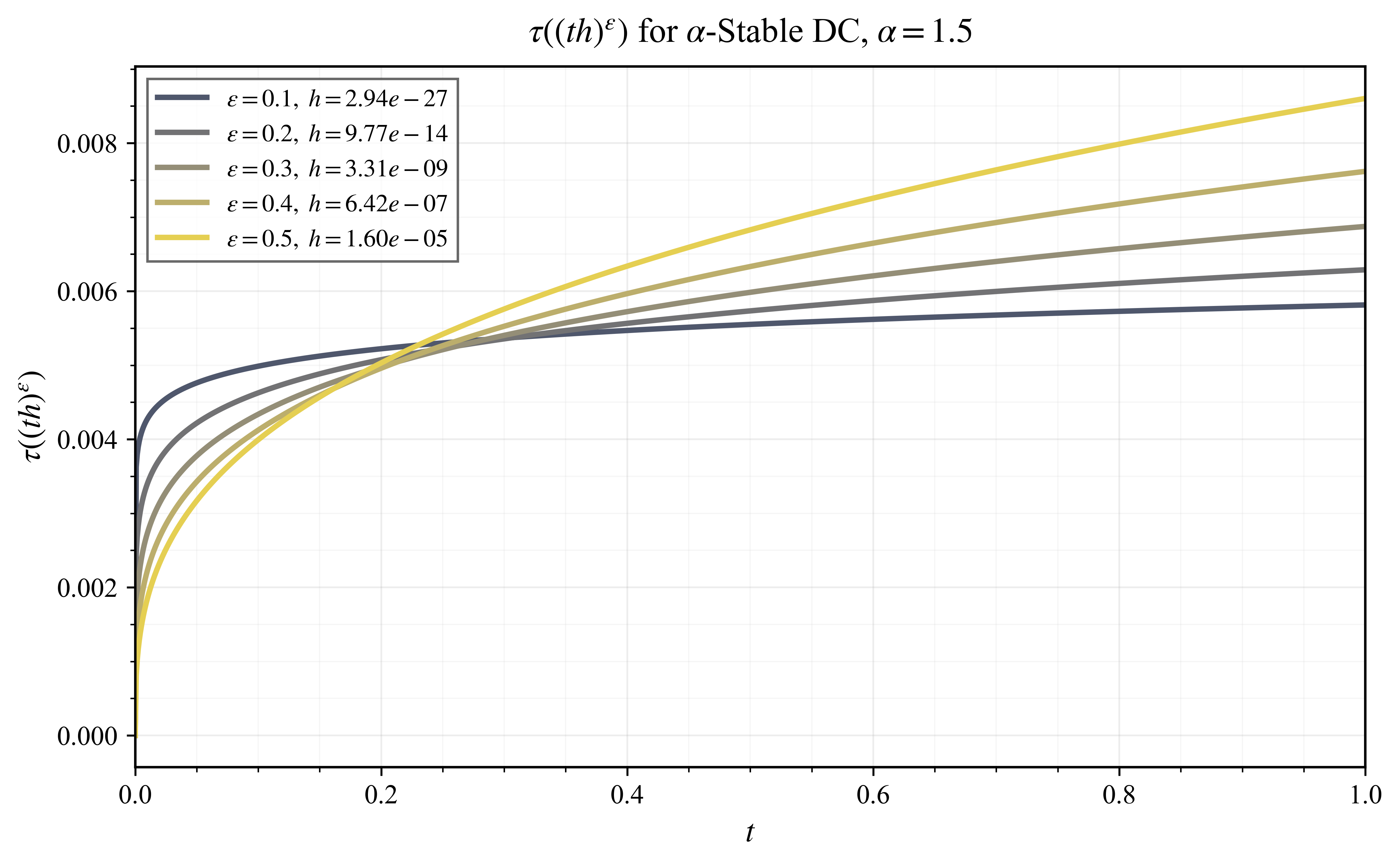}
    \caption{Visualization of the DC cutoff $\tau(t)$ for different $\varepsilon$.}
    \label{fig:tau-viz}
\end{figure}

For the $\alpha$-stable case we can calculate the cost explicitly
\[
\JJ(r_{c,\varepsilon})
=
\frac{2M}{1-\varepsilon}\,c^{-\alpha}T^{1-\varepsilon},
\qquad 0\le \varepsilon<1.
\]
Thus, for a fixed cost $J>0$, we choose
\[
c_\varepsilon(J)
=
\left(\frac{2M T^{1-\varepsilon}}{(1-\varepsilon)J}\right)^{1/\alpha}.
\]
This cutoff is non-truncated when
\[
c_\varepsilon(J)T^{\varepsilon/\alpha}\le 1
\iff
J\ge \frac{2MT}{1-\varepsilon}.
\]
Hence, for fixed cost $J$, the admissible non-truncated range is
\begin{equation}\label{eps1}
0\le \varepsilon \le 1-\frac{2MT}{J}.
\end{equation}

Let us also introduce two helper notations:
\[
\theta\coloneq \frac{2-\alpha}{\alpha}>0,
\qquad
K\coloneq \frac{2\alpha M}{2-\alpha}.
\]

\begin{theorem}\label{thm:matched-cost}
    Fix a cost $J>0$, and let $\varepsilon\in[0,1)$ be as in \eqref{eps1}.
    \begin{enumerate}[label=(\roman*),itemsep=2pt,topsep=4pt,beginpenalty=10000]
        \item If $1-2\sigma+\varepsilon\theta>0$, then the corresponding value of the exact mean-square error from \eqref{eq:ito-isometry-error} is
        \begin{equation}\label{eq:E-epsilon-cost}
            \EE_{\varepsilon}(J)
            = K\left(\frac{2M T^{1-\varepsilon}}{(1-\varepsilon)J}\right)^{\theta}
            \frac{T^{1-2\sigma+\varepsilon\theta}}{1-2\sigma+\varepsilon\theta}.
        \end{equation}

        \item If $\sigma<1/2$, then relative to the constant-cutoff case $\varepsilon=0$, the exact matched-cost variance ratio is
        \begin{equation}\label{eq:R-epsilon}
            R(\varepsilon) \coloneq \frac{\EE_{\varepsilon}(J)}{\EE_{0}(J)}
            = \frac{1-2\sigma}{1-2\sigma+\varepsilon\theta}(1-\varepsilon)^{-\theta}.
        \end{equation}

        \item If $\sigma<1/2$ and $J \ge 2MT / (1-\alpha\sigma)$, then $R$ is uniquely minimized over all admissible $\varepsilon\in\left[0,1-\frac{2MT}{J}\right]$ at $\varepsilon^\ast=\alpha\sigma\in[0,1)$.
    \end{enumerate}
\end{theorem}

An alternative approach to fair comparison between two different cutoff methods is to match the omitted small-jump variance. The following proposition demonstrates the dual nature of the problem.

Fix a target mean-square error level $v>0$. Over the same admissible non-truncated class as in Theorem~\ref{thm:matched-cost}, for each admissible exponent $\varepsilon\in[0,1)$ satisfying $\EE(r_{c,\varepsilon})=v$, denote by $J_\varepsilon(v)$ the corresponding jump cost.

\begin{proposition}\label{cor:matched-error}
    Suppose that $\sigma<1/2$. Then
    \begin{equation}\label{eq:Reff}
        R_{\mathrm{err}}(\varepsilon;v)
        \coloneq
        \frac{J_\varepsilon(v)}{J_0(v)}
        = \bigl(R(\varepsilon)\bigr)^{1/\theta}.
    \end{equation}
    If, in addition, for $\varepsilon^\ast=\alpha\sigma$ we have
    $J_{\varepsilon^\ast}(v)\ge \frac{2MT}{1-\varepsilon^\ast}$,
    then $R_{\mathrm{err}}$ is minimized over the same admissible non-truncated class at the same exponent
    $\varepsilon^\ast=\alpha\sigma$.
    In that regime, $R_{\mathrm{err}}(\varepsilon^\ast)<1$ whenever $\sigma>0$.
\end{proposition}

Finally, we show that given a cutoff level $r(t)$, the mean-square error $\EE(r)$ controls the distance between $p$-moments of $X_T$ and $X_T^{(r)}$.
Let
\begin{equation}\label{eq:weak-error-def}
    W_p(r) \coloneq \abs*{\E\abs{X_T}^p - \E\abs{X_T^{(r)}}^p}.
\end{equation}
For Theorem~\ref{thm:weak-bound} we consider cutoffs $r$ such that $\EE(r)<\infty$.

\begin{theorem}\label{thm:weak-bound}
    Let $0<p<2$, and suppose that $X_T,X_T^{(r)}\in L^p$. Then there exists a constant $C_p\in(0,\infty)$ depending only on $p$ such that
    \begin{equation}\label{eq:weak-bound}
        W_p(r) \le C_p\,\EE^{p/2}(r).
    \end{equation}
    In particular, for fixed computational cost, any cutoff that minimizes $\EE(r)$ also minimizes the derived upper bound on $W_p(r)$.
\end{theorem}

\section{Proof of Theorem~\ref{thm:global-opt}}\label{sec:universal-shape}

We begin with a few auxiliary lemmas. Let
\begin{equation}\label{eq:set-A}
    \mathcal A\coloneq
    \left\{
    u\in L^1([0,T]) :
    u(t)\ge N(1) \quad \text{a.e.},\quad
    \int_0^T u(t)\,\dd t = J_0
    \right\},
\end{equation}
and define
\[
\Phi(t)\coloneq V\bigl(N^{-1}(t)\bigr),\quad
\mathcal A_{\mathrm{fin}}\coloneq \{u\in\mathcal A: I(u)<\infty\}.
\]

\begin{lemma}\label{lem:reparam-convex}
The constrained problem
\[
\min_r \EE(r)
\qquad\text{subject to}\qquad
\JJ(r)=J_0
\]
over measurable cutoffs $r:[0,T]\to(0,1]$ is equivalent to the problem
\begin{equation}\label{Iu}
\min_{u\in\mathcal A} I(u),
\qquad I(u)\coloneq \int_0^T t^{-2\sigma}\Phi(u(t))\,\dd t \in [0,\infty].
\end{equation}
Moreover, $\Phi$ is strictly convex on $[N(1),\infty)$, and $I$ is strictly convex on
$\mathcal A_{\mathrm{fin}}$.
\end{lemma}

\begin{proof}
By Assumption~\ref{ass:global-opt}(A1), the function $N:(0,1]\to[N(1),\infty)$ is continuous and strictly decreasing.
Since $N(0+)=\infty$, it is bijective.
Define $u(t)\coloneq N(r(t))$. Clearly, $u$ is measurable, $u(t)\ge N(1)$ for all $t$, and by strict monotonicity of $N$ we have $r(t)=N^{-1}(u(t))$, and vice versa.
Under this change of variables, we can rewrite
\[
\JJ(r)=\int_0^T N(r(t))\,\dd t=\int_0^T u(t)\,\dd t,
\]
and
\[
\EE(r)
=
\int_0^T t^{-2\sigma}V(r(t))\,\dd t
=
\int_0^T t^{-2\sigma}V\bigl(N^{-1}(u(t))\bigr)\,\dd t
=
I(u).
\]
Hence the original problem is equivalent to minimizing the possibly infinite-valued functional $I$ over $\mathcal A$.

Let us show the convexity. Since
\[
V(\rho)=2 \int_0^\rho s^2\,\nu(\dd s)=\int_0^\rho s^2\bigl(-N'(s)\bigr)\,\dd s
\]
and $u\mapsto -u^2N'(u)$ is continuous on $(0,1]$, it follows that
\begin{equation}\label{V10}
V\in C^1((0,1])
\qquad\text{and}\qquad
V'(\rho)=-\rho^2N'(\rho),\qquad \rho\in(0,1].
\end{equation}
For $u>N(1)$, let $r=N^{-1}(u)$. Since $N\in C^1((0,1])$, the inverse function theorem and \eqref{V10} yield
\[
\Phi'(u)
=V'(r)(N^{-1})'(u)
=-r^2.
\]
Differentiating once more,
\[
\Phi''(u)
=-2r\,(N^{-1})'(u)
=-\frac{2r}{N'(r)}>0,
\]
because $r>0$ and $N'(r)<0$. Hence $\Phi$ is strictly convex on $(N(1),\infty)$. Since $N^{-1}(u)\uparrow 1$ as $u\downarrow N(1)$, the right derivative of $\Phi$ exists and equals
\[
\Phi_+'\bigl(N(1)\bigr)=\lim_{u\downarrow N(1)}\Phi'(u)=-1,
\]
so $\Phi$ is strictly convex on $[N(1),\infty)$.

Finally, the set $\mathcal A$ is convex, and $t^{-2\sigma}>0$ for $t\in(0,T]$. Therefore the strict convexity of $\Phi$ on $[N(1),\infty)$ implies that $I(u)$ is strictly convex on $\mathcal A_{\mathrm{fin}}$.
\end{proof}

\begin{lemma}\label{lem:V-power}
Let Assumption~\ref{ass:global-opt} hold. Then there exists a constant $C_V>0$ such that
\begin{equation}\label{eq:V-power-bound}
    V(\rho)\le C_V \rho^{2-\alpha},
    \qquad \rho\in(0,1].
\end{equation}
\end{lemma}

\begin{proof}
Fix $\rho\in(0,1]$. Using \eqref{V10} and Assumption~\ref{ass:global-opt}(A3), we get by integration by parts
\[
V(\rho)
=
-\rho^2N(\rho)+2\int_0^\rho sN(s)\,\dd s
\le
2\int_0^\rho sN(s)\,\dd s.
\]
Applying again Assumption~\ref{ass:global-opt}(A3), we conclude that
\begin{equation}\label{CV}
V(\rho)
\le
2C_\alpha \int_0^\rho s^{1-\alpha}\,\dd s
=
\frac{2C_\alpha}{2-\alpha}\rho^{2-\alpha}
= C_V \rho^{2-\alpha}.
\end{equation}
\end{proof}

\begin{lemma}\label{lem:cost-map}
Let Assumption~\ref{ass:global-opt} hold, and suppose that
\begin{equation}\label{eq:alpha-sigma}
    \alpha\sigma<1.
\end{equation}
For $c>0$, define
\[
r_c(t)\coloneq (ct^\sigma)\wedge 1,
\qquad
F(c)\coloneq \int_0^T N(r_c(t))\,\dd t.
\]
Then $F(c)<\infty$ and $\EE(r_c)<\infty$ for every $c>0$, and $F$ is continuous on $(0,\infty)$.
Moreover:
\begin{enumerate}[label=(\roman*),itemsep=2pt,topsep=4pt]
\item if $\sigma>0$, then $F$ is strictly decreasing on $(0,\infty)$;
\item if $\sigma=0$, then $F$ is strictly decreasing on $(0,1]$
      and constant on $[1,\infty)$.
\end{enumerate}
In all cases,
\[
\lim_{c\downarrow 0}F(c)=\infty,
\qquad
\lim_{c\to\infty}F(c)=TN(1).
\]
Consequently, for every $J_0>TN(1)$ there exists a unique constant $c^\ast>0$ such that
\[
F(c^\ast)=J_0.
\]
\end{lemma}

\begin{proof}
By Assumption~\ref{ass:global-opt}(A3), for every $t\in(0,T]$, $N(r_c(t))\le C_\alpha c^{-\alpha}t^{-\alpha\sigma}$
on the set $\{t\in(0,T]: ct^\sigma<1\}$, while $N(r_c(t))=N(1)$ on the set
$\{t\in(0,T]: ct^\sigma\ge 1\}$. Hence
\begin{equation}\label{eq:F-dom}
    N(r_c(t))\le N(1)\mathbf{1}_{\{ct^\sigma\ge 1\}}
    + C_\alpha c^{-\alpha}t^{-\alpha\sigma}\mathbf{1}_{\{ct^\sigma<1\}},
    \qquad t\in(0,T].
\end{equation}
Because $\alpha\sigma<1$,
$F(c)<\infty$ for every $c>0$.
Next, by Lemma~\ref{lem:V-power},
$V(r_c(t))\le C_V c^{2-\alpha}t^{\sigma(2-\alpha)}$
on the set $\{t\in(0,T]: ct^\sigma<1\}$, while $V(r_c(t))=V(1)$ on
$\{t\in(0,T]: ct^\sigma\ge 1\}$. Therefore,
\begin{align*}
t^{-2\sigma}V(r_c(t))
&\le
V(1)t^{-2\sigma} \mathbf{1}_{\{ct^\sigma\ge 1\}} + C_V c^{2-\alpha} t^{-\alpha\sigma}\mathbf{1}_{\{ct^\sigma<1\}}\\
&\le
V(1)c^2 \mathbf{1}_{\{ct^\sigma\ge 1\}} + C_V c^{2-\alpha} t^{-\alpha\sigma}\mathbf{1}_{\{ct^\sigma<1\}}.
\end{align*}
Again using $\alpha\sigma<1$, we conclude that $\EE(r_c)<\infty$ for every $c>0$.

Fix $c_0>0$. For $c\in[c_0/2,2c_0]$,
\[
N(r_c(t))
\le N(1)+C_\alpha (c_0/2)^{-\alpha} t^{-\alpha\sigma},
\qquad t\in(0,T],
\]
and the right-hand side is integrable because $\alpha\sigma<1$. Since $r_c(t)\to r_{c_0}(t)$ pointwise, dominated convergence yields continuity of $F$ at $c_0$, hence on $(0,\infty)$.

Further, since $c\mapsto r_c(t)$ is increasing and $N$ is decreasing, we see that $c\mapsto F(c)$ is decreasing; in particular, it is strictly decreasing if $\sigma>0$. If $\sigma=0$, then $r_c(t)=c\wedge 1$ for all $t$, and therefore $F(c)=TN(c\wedge 1)$, i.e. in this case $F$ is strictly decreasing on $(0,1]$ and constant on $[1,\infty)$.

Let us show the last statement of the lemma. As $c\to\infty$, we get by monotonicity $r_c(t)\uparrow 1$, hence $N(r_c(t))\downarrow N(1)$, and thus dominated convergence yields
\[
\lim_{c\to\infty}F(c)=\int_0^T N(1)\,\dd t=TN(1).
\]
As $c\downarrow 0$, choose $c$ so small that $cT^\sigma\le 1$.
Then on $[0,T]$, $r_c(t)=ct^\sigma$, and therefore
\[
F(c)=\int_0^T N(ct^\sigma)\,\dd t\ge \int_0^T N(cT^\sigma)\,\dd t = T N(cT^\sigma).
\]
Using again monotonicity, we see that $\lim_{c\downarrow 0}F(c)=\infty$, which proves the existence of a solution to $F(c^\ast)=J_0$ for every $J_0>TN(1)$.
\end{proof}

\begin{proof}[Proof of Theorem~\ref{thm:global-opt}]
Recall the definition of $\mathcal A$ from \eqref{eq:set-A}.
For part (i), if $u\in\mathcal A$, then $u(t)\ge N(1)$ a.e. and
\[
\int_0^T \bigl(u(t)-N(1)\bigr)\,\dd t=0,
\]
implying that $u=N(1)$ a.e. Thus $\mathcal A$ contains only one element,
namely $u^\ast\equiv N(1)$, and the corresponding cutoff is
\[
r^\ast=N^{-1}\bigl(N(1)\bigr)=1
\qquad\text{on }[0,T].
\]
Because $\sigma<1/2$, we have
\[
\EE(r^\ast)=V(1)\int_0^T t^{-2\sigma}\,\dd t<\infty,
\]
so this cutoff is the unique minimizer.

We now prove part (ii). By Lemma~\ref{lem:cost-map}, there exists a unique $c^\ast>0$
such that
\[
\int_0^T N\bigl((c^\ast t^\sigma)\wedge 1\bigr)\,\dd t=J_0.
\]
Define
\[
r^\ast(t)\coloneq (c^\ast t^\sigma)\wedge 1,
\qquad
u^\ast(t)\coloneq N(r^\ast(t)).
\]
Then $u^\ast\in\mathcal A$, and Lemma~\ref{lem:cost-map} yields $\EE(r^\ast)<\infty$.

By Lemma~\ref{lem:reparam-convex}, the original problem is equivalent to minimizing $I(u)$ from \eqref{Iu} over $\mathcal A$.

Let $u\in\mathcal A$ be arbitrary. If $I(u)=\infty$, then automatically
$I(u)\ge I(u^\ast)$, so there is nothing to prove. Assume from now on that $I(u)<\infty$.

Because the (right) derivative of $\Phi$ at $u^\ast(t)$ equals $-(r^\ast(t))^2$,
the convexity of $\Phi$ yields
\[
\Phi(u(t))-\Phi(u^\ast(t))\ge -(r^\ast(t))^2\bigl(u(t)-u^\ast(t)\bigr)
\qquad\text{for a.e. }t\in[0,T].
\]
Since both $u,u^\ast\in\mathcal A$, they have the same equality constraint, so
\[
\int_0^T \bigl(u(t)-u^\ast(t)\bigr)\,\dd t=0.
\]
Multiplying by $t^{-2\sigma}$, integrating, and adding the zero term
$\int_0^T (c^\ast)^2 \bigl(u(t)-u^\ast(t)\bigr)\,\dd t$, we obtain
\[
I(u)-I(u^\ast)
\ge
\int_0^T \left((c^\ast)^2-t^{-2\sigma}(r^\ast(t))^2\right)\bigl(u(t)-u^\ast(t)\bigr)\,\dd t.
\]
On the set where $r^\ast(t)<1$, we have $r^\ast(t)=c^\ast t^\sigma$, and therefore
$(c^\ast)^2-t^{-2\sigma}(r^\ast(t))^2=0$.

On the set where $r^\ast(t)=1$, we have $c^\ast t^\sigma\ge 1$; consequently,
$(c^\ast)^2-t^{-2\sigma}(r^\ast(t))^2\ge 0$.
Moreover, on this same set,
$u^\ast(t)=N(1)$,
while every $u\in\mathcal A$ satisfies $u(t)\ge N(1)$ a.e. Hence
\[
u(t)-u^\ast(t)\ge 0
\qquad\text{for a.e. }t\text{ with }r^\ast(t)=1.
\]
Therefore
\[
\left((c^\ast)^2-t^{-2\sigma}(r^\ast(t))^2\right)\bigl(u(t)-u^\ast(t)\bigr)\ge 0
\qquad\text{for a.e. }t\in[0,T],
\]
implying that
\[
I(u)\ge I(u^\ast)
\qquad\text{for every }u\in\mathcal A.
\]

Thus, $u^\ast$ is a global minimizer of $I$ on $\mathcal A$.

Since $I$ is strictly convex on $\mathcal A_{\mathrm{fin}}$, it has at most one minimizer there. As $u^\ast\in\mathcal A_{\mathrm{fin}}$ and $I(u)\ge I(u^\ast)$ for every $u\in\mathcal A$, while functions with $I(u)=\infty$ cannot minimize, it follows that $u^\ast$ is the unique minimizer on $\mathcal A$. Therefore the corresponding cutoff $r^\ast=N^{-1}(u^\ast)$ is unique.

Finally, if $c^\ast T^\sigma\le 1$, then $(c^\ast t^\sigma)\wedge 1=c^\ast t^\sigma$,
so $r^\ast(t)=c^\ast t^\sigma$ for all $t\in(0,T]$.
\end{proof}

\begin{remark}\label{rem:global-opt-extended}
Under Assumption~\ref{ass:global-opt}(A3), part~(ii) of Theorem~\ref{thm:global-opt}
extends the optimization result to the range $\alpha\sigma<1$,
which may allow $\sigma\ge 1/2$. The boundary case $J_0=TN(1)$ is kept only under
$\sigma<1/2$, because then the unique feasible cutoff is $r\equiv 1$, while for
$\sigma\ge 1/2$ one generally has
\[
\EE(1)=V(1)\int_0^T t^{-2\sigma}\,\dd t=\infty.
\]
\end{remark}

\section{Proof of Theorem~\ref{thm:matched-cost} and Proposition~\ref{cor:matched-error}}\label{sec:stable}

\begin{proof}[Proof of Theorem~\ref{thm:matched-cost}]
\emph{Parts~(i)--(ii).} If $1-2\sigma+\varepsilon\theta>0$, then
\[
\EE(r_{c,\varepsilon}) = K c^{2-\alpha}\int_0^T t^{-2\sigma+\varepsilon(2-\alpha)/\alpha}\,\dd t
= K c^{2-\alpha}\frac{T^{1-2\sigma+\varepsilon\theta}}{1-2\sigma+\varepsilon\theta}.
\]
Substituting $c = \left(\frac{2M T^{1-\varepsilon}}{(1-\varepsilon)J}\right)^{1/\alpha}$ proves \eqref{eq:E-epsilon-cost}. Since $\sigma<1/2$, taking the quotient with $\varepsilon=0$ yields \eqref{eq:R-epsilon}.

\emph{Part~(iii).} To optimize, we note that $\log R(\varepsilon)$ is strictly convex, as its second derivative is strictly positive. Setting the first derivative to zero yields
\begin{equation}\label{eq:logR-derivative}
    \frac{\dd}{\dd\varepsilon}\log R(\varepsilon) = \frac{\theta}{1-\varepsilon} - \frac{\theta}{1-2\sigma+\varepsilon\theta} = 0.
\end{equation}
Since $1+\theta=2/\alpha>0$, we obtain the unique minimizer $\varepsilon^\ast=\alpha\sigma$.
Under the condition $J \ge \frac{2MT}{1-\alpha\sigma}$, this critical point belongs to the admissible interval $\left[0,1-\frac{2MT}{J}\right]$.

Finally, if $0<\sigma<1/2$, then $R(\varepsilon^\ast)<1=R(0)$, while if $\sigma=0$, then $\varepsilon^\ast=0$ and the DC cutoff coincides with the fixed one.
\end{proof}

\begin{remark}
The dependence on the parameters and the magnitude can be seen in Figure~\ref{fig:R-viz}. The stronger the singularity, the more pronounced the advantage of DC over fixed cutoff.
\end{remark}

\begin{figure}[htbp]
    \centering
    \includegraphics[width=0.85\textwidth]{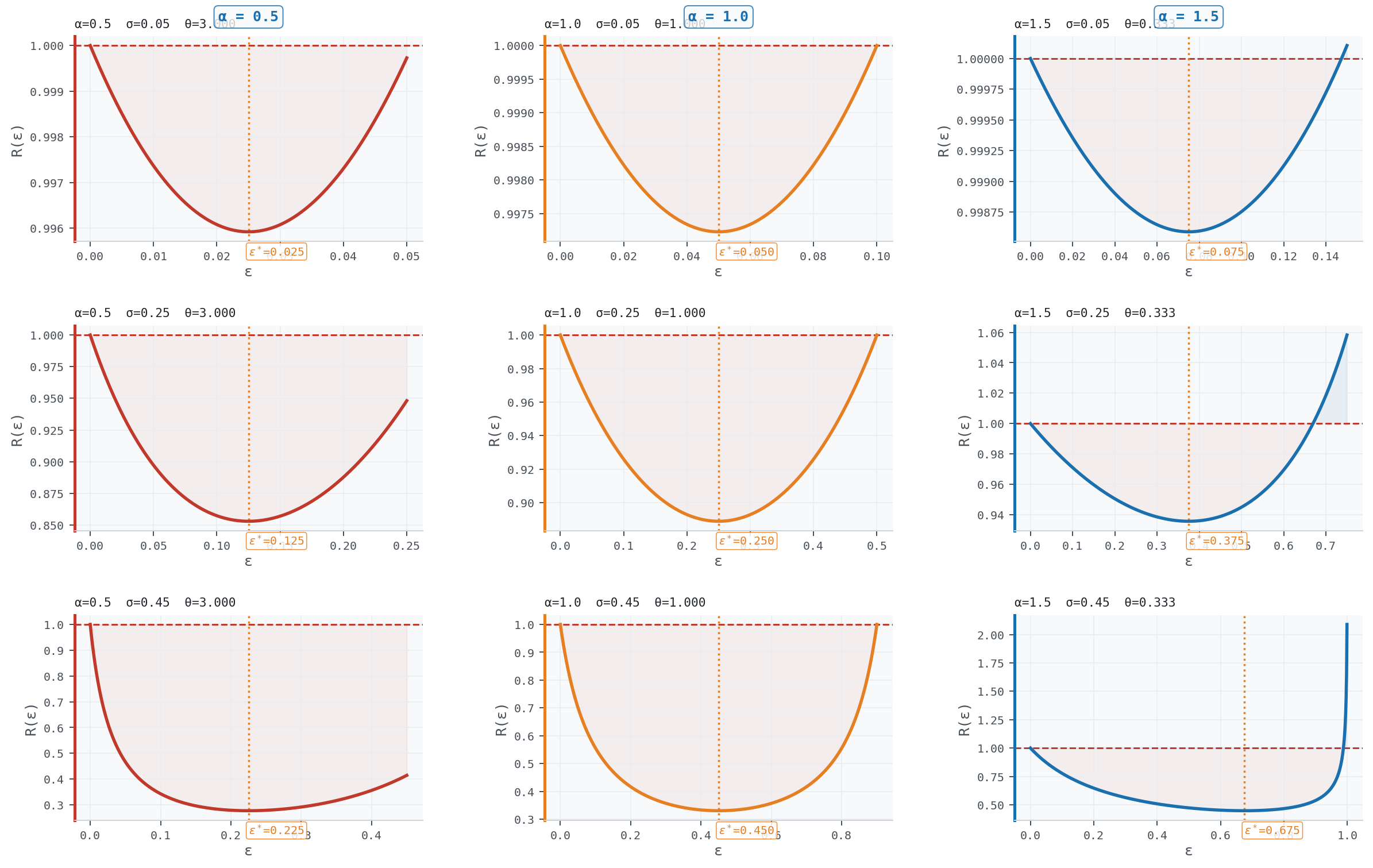}
    \caption{Visualization of $R(\varepsilon)$.}
    \label{fig:R-viz}
\end{figure}

\begin{proof}[Proof of Proposition~\ref{cor:matched-error}]
Recall from the proof of Theorem~\ref{thm:matched-cost} that
$\EE(r_{c,\varepsilon})
=
\frac{K c^{2-\alpha} T^{1-2\sigma+\varepsilon\theta}}{1-2\sigma+\varepsilon\theta}$.
Let $v =
K c^{2-\alpha}\frac{T^{1-2\sigma+\varepsilon\theta}}{1-2\sigma+\varepsilon\theta}$.
Since $2-\alpha=\alpha\theta$, solving $\EE(r_{c,\varepsilon})=v$ gives
\[
c^{-\alpha}
=
\left(
 \frac{K}{v} \frac{T^{1-2\sigma+\varepsilon\theta}}{1-2\sigma+\varepsilon\theta}
\right)^{1/\theta},
\]
which, after substituting into $\JJ(r_{c,\varepsilon})
=
\frac{2M}{1-\varepsilon}c^{-\alpha}T^{1-\varepsilon}
$ (cf. \eqref{eq:cost-error-functionals})
yields
\[
J_\varepsilon(v)
=
\frac{2M}{1-\varepsilon}T^{1-\varepsilon}
\left(
\frac{K}{v} \frac{T^{1-2\sigma+\varepsilon\theta}}{1-2\sigma+\varepsilon\theta}
\right)^{1/\theta}.
\]
Dividing by the corresponding expression for $\varepsilon=0$, we obtain
\[
R_{\mathrm{err}}(\varepsilon;v)
=
\left[
\frac{1-2\sigma}{1-2\sigma+\varepsilon\theta}(1-\varepsilon)^{-\theta}
\right]^{1/\theta}
=
\bigl(R(\varepsilon)\bigr)^{1/\theta}.
\]
Since $x\mapsto x^{1/\theta}$ is strictly increasing on $(0,\infty)$, the minimizer over the same admissible non-truncated class is the same as for $R$. The final claim follows from Theorem~\ref{thm:matched-cost} on that same admissible class.
\end{proof}

\section{Proof of Theorem~\ref{thm:weak-bound}}\label{sec:weak}

We begin with an auxiliary lemma.

\begin{lemma}\label{lem:phi}
Let $\varphi_Y(u)\coloneq \E[e^{iuY}]$ be the characteristic function of a random variable $Y$. For $0<p<2$, assume that $Y\in L^p$. Then there exists a constant $c_p>0$ such that
\[
\E\abs{Y}^p = c_p\int_{\R}\frac{1-\operatorname{Re}\bigl(\varphi_Y(s)\bigr)}{\abs{s}^{1+p}}\,\dd s.
\]
\end{lemma}

\begin{proof}
By definition, the real part of the characteristic function is $\operatorname{Re}\bigl(\varphi_Y(s)\bigr) = \E[\cos(sY)]$. Since $1-\cos(sY) \ge 0$, Tonelli's theorem allows us to interchange the integral and the expectation:
\begin{align*}
\int_{\R}\frac{1-\operatorname{Re}\bigl(\varphi_Y(s)\bigr)}{\abs{s}^{1+p}}\,\dd s
&= \E\left[\int_{\R} \frac{1-\cos(sY)}{\abs{s}^{1+p}}\,\dd s\right].
\end{align*}
For $Y\neq 0$, making the change of variables $v=sY$ yields
\[
\int_{\R}\frac{1-\cos(sY)}{\abs{s}^{1+p}}\,\dd s
=
\abs{Y}^p \int_{\R}\frac{1-\cos(v)}{\abs{v}^{1+p}}\,\dd v.
\]
The same identity holds when $Y=0$, since both sides are zero.
Thus,
\[
\int_{\R}\frac{1-\operatorname{Re}\bigl(\varphi_Y(s)\bigr)}{\abs{s}^{1+p}}\,\dd s
=
\E \abs{Y}^p \int_{\R} \frac{1-\cos(v)}{\abs{v}^{1+p}}\,\dd v.
\]
The last integral can be calculated explicitly using $1-\cos(v) = 2\sin^2(v/2)$ and \cite[(3.823)]{GradshteynRyzhik2015}:
\[
\int_{0}^{\infty} x^{\mu-1} \sin^2(ax)\,\dd x = -\frac{\Gamma(\mu) \cos \frac{\mu\pi}{2}}{2^{\mu+1}a^\mu}
\qquad a > 0, \quad -2 < \operatorname{Re} \mu < 0.
\]
Applying this formula with $\mu = -p$, and $a = 1/2$,
\[
c_p^{-1} \coloneq \int_{\R} \frac{1-\cos(v)}{\abs{v}^{1+p}}\,\dd v
= 4 \int_{0}^{\infty} \frac{\sin^2(v/2)}{v^{1+p}}\,\dd v
= \frac{\pi}{p\,\Gamma(p)\sin(p\pi/2)}.
\]
Substituting this constant back into the expectation yields the desired result.
\end{proof}

\begin{proof}[Proof of Theorem~\ref{thm:weak-bound}]
Denote by $\varphi_Y(u)\coloneq \E[e^{iuY}]$ the characteristic function of $Y$.
Then applying Lemma~\ref{lem:phi} to $X_T$ and $X_T^{(r)}$ and using the triangle inequality, we derive
\begin{equation}\label{eq:weak-cf-bound}
    W_p(r) \le c_p\int_{\R}\frac{\abs{\varphi_{X_T^{\vphantom{(r)}}}(s)-\varphi_{X_T^{(r)}}(s)}}{\abs{s}^{1+p}}\,\dd s.
\end{equation}
Note that the random variables $X_T^{(r)}$ and $R_T^{(r)}$ are independent: they are defined by restricting the Poisson random measure to the disjoint deterministic sets $\{(t,z):0<t\le T,\ |z|>r(t)\}$ and $\{(t,z):0<t\le T,\ |z|\le r(t)\}$, respectively, and deterministic compensation does not affect independence. Therefore, since
$\varphi_{X_T^{\vphantom{(r)}}}(s)=\varphi_{X_T^{(r)}}(s)\varphi_{R_T^{(r)}}(s)$,
then
\[
\abs{\varphi_{X_T^{\vphantom{(r)}}}(s)-\varphi_{X_T^{(r)}}(s)}
= \abs{\varphi_{X_T^{(r)}}(s)}\,\abs{1-\varphi_{R_T^{(r)}}(s)}
\le \abs{1-\varphi_{R_T^{(r)}}(s)}.
\]
Note that the residual $R_T^{(r)}$ is symmetric because $\nu$ is symmetric. Then
\[
\varphi_{R_T^{(r)}}(s)=\E[\cos(sR_T^{(r)})]
\]
Using $1-\cos y \le \min(2,y^2/2)$, we obtain
\[
1-\varphi_{R_T^{(r)}}(s)=\E[1-\cos(sR_T^{(r)})] \le \min\left(2,\frac{s^2}{2}\E[(R_T^{(r)})^2]\right).
\]
By \eqref{eq:ito-isometry-error},
$\E[(R_T^{(r)})^2]=\EE(r)$. Substituting into \eqref{eq:weak-cf-bound} yields
\[
W_p(r) \le c_p\int_{\R} \min\left(2,\frac{s^2}{2}\EE(r)\right)\frac{\dd s}{\abs{s}^{1+p}}.
\]
Assume from now on that $\EE(r)>0$, otherwise there is nothing to prove.
Set
$u_0\coloneq 2/\sqrt{\EE(r)}$. Then
\[
\min\left(2,\frac{s^2}{2}\EE(r)\right)
=
\begin{cases}
\frac{s^2}{2}\EE(r), & \abs{s}\le u_0,\\
2, & \abs{s}>u_0.
\end{cases}
\]
Therefore,
\begin{align*}
    W_p(r)
    &\le 2c_p\left[
    \int_0^{u_0}\frac{(s^2/2)\EE(r)}{s^{1+p}}\,\dd s
    +
    \int_{u_0}^{\infty}\frac{2}{s^{1+p}}\,\dd s
    \right]\\
    &= 2c_p\left[
    \frac{\EE(r)}{2}\cdot \frac{u_0^{2-p}}{2-p}
    +
    2\cdot \frac{u_0^{-p}}{p}
    \right]
    = C_p\,\EE^{p/2}(r),
\end{align*}
where we used $u_0^{2-p}=2^{2-p}\EE^{-(2-p)/2}(r)$ and
$u_0^{-p}=2^{-p}\EE^{p/2}(r)$, and absorbed the resulting finite constant into $C_p$:
\[
C_p \coloneq c_p\,2^{2-p}\left(\frac{1}{2-p}+\frac{1}{p}\right).
\]
This proves \eqref{eq:weak-bound}.
\end{proof}

\begin{corollary}\label{cor:weak-stable}
Let $p\in(0,\alpha)$. If $\sigma<1/2$ and the hypotheses of Theorem~\ref{thm:weak-bound} hold for this $p$, then in the non-truncated admissible matched-cost regime of Theorem~\ref{thm:matched-cost} we have
\begin{equation}\label{eq:weak-ratio-bound}
    W_p\bigl(r_{c_{\varepsilon}(J),\varepsilon}\bigr)
    \le C_p\,\EE_0^{p/2}(J) \,R^{p/2}(\varepsilon).
\end{equation}
Therefore the analytic factor $R^{p/2}(\varepsilon)$ is minimized at
$\varepsilon^\ast=\alpha\sigma$; if $\varepsilon^\ast$ is admissible in the non-truncated regime for the chosen cost $J$,
then it minimizes the derived weak-error upper bound within the non-truncated admissible matched-cost
family.
\end{corollary}

\begin{proof}
Combine Theorem~\ref{thm:weak-bound} with \eqref{eq:R-epsilon}. Since $x\mapsto x^{p/2}$ is strictly increasing on $(0,\infty)$, the analytic minimizer is the same as for $R$. The final claim concerns the upper bound \eqref{eq:weak-ratio-bound}, not the exact weak error.
\end{proof}

\medskip
\noindent\textbf{Source code.}
The source code used for the calculations and to generate Figures~\ref{fig:tau-viz} and~\ref{fig:R-viz} is available at \url{https://github.com/d-platonov/CMP}.

\end{document}